\documentclass[preprint,12pt]{elsarticle}

\usepackage{amssymb,latexsym,amsmath}     
\usepackage{epsf}
\usepackage{array}
\usepackage{amssymb,amscd}
\usepackage{calc}
\usepackage[mathscr]{eucal}
\usepackage{hhline}
\usepackage{tikz}
\usepackage{multirow}

\usepackage{amssymb,amsmath,amstext,amsgen,amsfonts,amsthm,mathrsfs,verbatim,url,hyperref,mathdots}
\usepackage{enumitem}

\newlist{inparaenum}{enumerate}{2}
\setlist[inparaenum]{nosep}
\setlist[inparaenum,1]{label=\bfseries\alph*.}

\newtheorem{theorem}{Theorem}[section]
\newtheorem{definition}[theorem]{Definition}
\newtheorem{lemma}[theorem]{Lemma}
\newtheorem{proposition}[theorem]{Proposition}

\def\comment#1{}
 \DeclareMathOperator\diag{diag}

\def\invddots{\mathinner{\mskip1mu\raise1pt\vbox{\kern7pt\hbox{.}}\mskip2mu
		\raise4pt\hbox{.}\mskip2mu\raise7pt\hbox{.}\mskip1mu}}

\journal{TBD}

\begin{document}

\begin{frontmatter}



\title{Lipschitz stability of $\gamma$-FOCS and RC canonical Jordan bases of real $H$-selfadjoint  matrices under small perturbations}


\author[label2]{S. Dogruer Akgul}
\author[label2]{A. Minenkova}
\author[label2]{V. Olshevsky}


\address[label2]{
	{Department of Mathematics,  University of Connecticut, Storrs CT 06269-3009,  USA. Email: sahinde.dogruer@uconn.edu, anastasiia.minenkova@uconn.edu, olshevsky@uconn.edu}}

\begin{abstract}
In 2008 Bella, Olshevsky and Prasad proved that the flipped orthogonal (FO) Jordan bases of $H$-selfadjoint matrices  are Lipschitz stable under small perturbations. In 2022, Dogruer, Minenkova and Olshevsky considered the real case, and proved that for real $H$-selfadjoint matrices there exist a more refined bases called FOCS bases. In addition to flipped orthogonality they also possess the conjugate symmetric (CS) property.  In this paper we prove that these new FOCS bases are  Lipschitz stable under small perturbations as well. We also establish the Lipschitz stability for the classical real canonical Jordan bases.
\end{abstract}


                                 \end{frontmatter}


\section{Introduction}

\subsection{$H$-selfadjoint matrices and the affiliation relation}

Let us start with the definition of {\bf the indefinite inner product}. It is a relation denoted by $[\cdot,\cdot]$ that satisfies all the axioms of the usual inner product except for positivity.

For example,  every $n\times n$ invertible Hermitian matrix $H$
determines an indefinite inner product through
\begin{equation}\label{Hinnerproduct}
    [x,y]_H=y^*Hx,\,\text{ for }x,y\in\mathbb{C}^n.
\end{equation} The converse is also true: for any indefinite inner product in $\mathbb{C}^n$ we can find such invertible Hermitian $H$ that \eqref{Hinnerproduct} holds true. If $H$ is positive-definite then $[\cdot,\cdot]_H$ turns into the classical inner product. However, in this paper we do not impose such an assumption.

{\bf H-selfadjoint matrices.} A matrix $A$ is called \textbf{H-selfadjoint} if 
\begin{equation}\label{HSA}
     A=H^{-1}A^*H\text{ (equivalently $[Ax,y]_H=[x,Ay]_H$)}.
\end{equation}
 There are many application for such matrices and they  have been studied by many authors (e.g. see \cite{GLR83} and many references therein).  
 
If follows from \eqref{HSA} that the eigenvalues of $H$-selfadjoint matrices are  symmetric with respect to the real axis. In addition, the conjugate pairs of eigenvalues have the  Jordan blocks of the same sizes.

{\bf Affiliation relation.  } For any change of basis    $$x\mapsto u =T^{-1}x, \quad     y\mapsto v =T^{-1}y,$$
we have $$[x,y]_H=y^*Hx=(Tv)^*H(Tu)=v^*\overset{G}{\overbrace{(T^*HT)}}u=[u,v]_G.$$
That is,
a new congruent matrix $G = T^*HT$  generates  the same inner product  in the new basis.


Two pairs $(A,H)$ and $(B,G)$ are called {\bf affiliated} if $$T^{-1}AT=B\text{ and }T^{*}HT=G.$$
and this equivalence relation is called the {\bf affiliation relation}.
\begin{equation}\label{AFR}
    (A,H)\mapsto(B,G).
\end{equation}
The affiliation relation preserves selfadjointness in the indefinite inner product, that is if $A$ is $H$-selfadjoint, then $B$ is $G$-selfadjoint. 
	
	For an $H_0$-selfadjoint matrix there exists an invertible matrix $T_0$  (see e.g.~\cite{BOP,GLR83,M63} and the references therein), such that $(A_0, H_0) \overset{T_0}{\mapsto} (J,P)$, i.e. 
\begin{equation*}\label{affiliation}
T_0^{-1}A_0T_0=J \qquad \text{and} \qquad T_0^{*}H_0T_0=P,
\end{equation*}
where $J$ is the Jordan Form of $A_0$ and $P$ has a particularly sparse simple form called {\it the sip matrix} in \cite[Section 5.5]{GLR83}, which is uniquely defined up-to perturbation of the Jordan blocks in $J$. This is a very special affiliation relation. It first appeared in the works by Weierstrass \cite{W68,W95} (see also \cite[Chapter 5]{GLR83} and \cite[Chapter 3]{M63}). The columns of $T_0$ form a Jordan basis of $A_0$, and this basis was called \lq\lq canonical\rq\rq 
(see \cite{GLR86}) and the pair $(J,P)$ is sometimes called the {\it Weierstrass form} of $A_0$. Since we introduce two more canonical bases, we suggest to call this basis {\it flipped-orthogonal} (or just FO) basis  of $(A_0, H_0)$ in order to distinguish  them.

	
\subsection{FO, CS, and FOCS Bases}

Here is an example of an $H$-selfadjoint matrix.
\begin{equation}\label{A}
    A=\begin{bmatrix}
0&0&0&-16\\
1&0&0&0\\
0&1&0&-8\\
0&0&1&0\end{bmatrix}=TJT^{-1},\quad    H=\frac{1}{128}\begin{bmatrix}
0&1&0&-12\\
1&0&-12&0\\
0&-12&0&80\\
-12&0&80&0\end{bmatrix},
\end{equation}
where  $A=H^{-1}A^*H$ and
\begin{equation}\label{JT}
    J=\left[\begin{array}{cc|cc}
-2i&1&0&0\\
0&-2i&0&0\\\hline
0&0&2i&1\\
0&0&0&2i\end{array}\right],\quad     P=\left[\begin{array}{cc|cc}
0&0&0&1\\
0&0&1&0\\\hline
0&1&0&0\\
1&0&0&0\end{array}\right],
\end{equation}
{and}
\begin{equation}\label{P}
T=\begin{bmatrix}
8i&4&-8i&12\\
-4&4i&-4&-8i\\
2i&-1&-2i&1\\
-1&0&-1&-i\end{bmatrix}.
\end{equation}

  Here, $(J,P)$ is the FO form of $(A,H)$ and  the columns of $T=[t_i]$ satisfy 
$$[t_i,t_j]_H=[e_i,e_j]_P=\delta_{i,4-j},$$
where $\delta_{i,j}$ is the Kronecker symbol. Therefore, the Jordan basis $\{t_1,\ldots,t_n\}$ is called \textbf{flipped orthogonal} in agreement with its usage in \cite{BOP}.


Note that FO bases exist in the general complex case. However, the matrices $(A,H)$ in our example are real, which suggests the following question.
 
 {\it Are there any canonical bases specific for the  case of real matrices?} Yes, for example a conjugate symmetric (CS) basis. Let us return to the same matrices $A$ and $H$ are as in \eqref{A}.
and consider another Jordan basis for $A$, captured by the columns of the following matrix.
\begin{equation}\label{L}
L=\begin{bmatrix}
8i&4&-8i&4\\
-4&4i&-4&-4i\\
2i&-1&-2i&-1\\
-1&0&-1&0\end{bmatrix}.
\end{equation}
The columns $L=[l_1\ l_2\ l_3\ l_4]$ are conjugate of each other:
$$l_1=\overline{l}_3\text{ and }l_2=\overline{l}_4.$$
Therefore, we say that  $L$ captures a {\bf conjugate symmetric basis} of $A$ from \eqref{A} and $(A,H)\overset{L}{\mapsto}(J,G)$, where
$$
G=L^*HL=\begin{bmatrix}
0&0&0&1\\
0&0&1&-i\\
0&1&0&0\\
1&i&0&0\end{bmatrix}
$$is not a sip matrix as was $P$ in \eqref{P}. This means $L$ is CS but not FO.
Also, if we look back at the FO basis in $T$ from \eqref{JT}, it is  not CS.

Now that we have seen the example of  the FO and CS canonical bases we started to wonder:
{\it Could we find a Jordan basis for real $H$-selfadjoint matrices that is   flipped orthogonal and $\gamma$-conjugate symmetric at the same time?} 

Here is a matrix whose columns also capture a Jordan basis of the same $A$ as before.
$$M=\frac{1}{2}\begin{bmatrix}
16i&16&-16i&16\\
-8&12i&-8&-12i\\
4i&0&-4i&0\\
-2&i&-2&-i\end{bmatrix},$$
such that  $(A,H)\overset{M}{\mapsto}(J,P)$ and $M$ is CS. So $M$ captures a \textbf{flipped orthogonal conjugate symmetric basis}.

It is only an example. 

{\it Can one guarantee the existence of such bases?} The positive answer to this question can be found in \cite{DMO}.

\subsection{Lipschitz Stability of FO Bases}
Stability of Jordan bases were studied previously, for example, in the paper by Bella,   Olshevsky and Prasad \cite{BOP}. They proved the Lipschitz 
 stability of the so-called flipped orthogonal (FO) Jordan bases of {\em complex} $H$-selfadjoint matrices under small perturbations. 
 


The following fact was presented in~\cite[Proposition 1.5]{BOP}.


\begin{proposition} {\bf{(Lipschitz stability of FO-bases)} }\label{thm:stabilityFO}Let $ A_0 \in \mathbb{C}^{ n\times n}$ be a fixed $H_0$-selfadjoint matrix and let the columns of this matrix capture the FO-basis, that is, 
	\begin{displaymath}
	(A_0,H_0)\overset{T_0}{\mapsto}(J,P).
	\end{displaymath}
	There exist constants $K,\delta > 0$ (depending on  $A_{0}$ and  $H_0$ only) such that  the following assertion holds for any $H$-selfadjoint matrix $A$ with $A$ similar to $A_0$ and
	\begin{displaymath}
	\| A-A_0\| + \| H-H_0\| < \delta,
	\end{displaymath}
	there exists a matrix $T$ whose columns capture the FO-basis of $A$, i.e.  
	\begin{displaymath}
	(A,H)\overset{T}{\mapsto}(J,P)
	\end{displaymath}and such that 
	\begin{equation}\label{perturbation2}
	\| T - T_{0} \| \leq K(\| A-A_0 \| + \| H-H_0 \|).
	\end{equation}
\end{proposition}





\subsection{Main results and the structure of the paper}
\subsubsection{Lipschitz Stability of FOCS Bases}

We are going to build on what we mentioned before and show that if the pair $(A,H)$ is real, and if matrix $T_0$ captures the FO but FOCS basis of $(A,H)$, there is a matrix $T$ which can  be chosen to capture the $\gamma$-FOCS basis of $(A,H)$ and still be Lipschitz stable under small perturbation preserving the Jordan structure. That is, we prove the bound \eqref{perturbation2}, in which both matrices $T_0$ and $T$ enjoy the additional CS-property.

\subsubsection{Lipshitz stability of RC bases}
We refer the reader to \cite{GLR83} for an introduction to real canonical bases. 

Let us come back to our example from \eqref{A}. Here is the real Jordan form $J_R$ of $A$ and its real canonical basis such that $R^*HR=P$.
$$
J_R=R^{-1}AR=\left[\begin{array}{cc|cc}
0&-2&1&0\\
2&0&0&1\\\hline
0&0&0&-2\\
0&0&2&0
\end{array}\right]\text{ and }
R=\begin{bmatrix}
-8&8&8&8\\
-4&-4&-6&6\\
-2&2&0&0\\
-1&-1&-0.5&0.5
\end{bmatrix}.
$$

In this paper, via FOCS bases we were able to obtain results on stability of {\it real Jordan canonical forms}. 

We prove that there exists an invertible matrix $S_0$ such that 
\begin{displaymath}
S_0^{-1}A_0S_0=J_R \qquad \text{and} \qquad S_0^{*}H_0S_0=P,
\end{displaymath}
where $J_R$ is a real Jordan form of $A_0$ (see Theorem~\ref{thm:real main} for details) and the columns of $S_0$ capture a {\it Real Canonical (RC) basis}. Moreover, if ${A} \in\mathbb{R}^ {n \times n}$  is a perturbation of $A_0$, preserving the Jordan structure, that is $H$-selfadjoint for some self-adjoint real matrix $H$, then there exists an invertible real matrix $S$ such that 
\begin{displaymath}
S^{-1}AS=J_R \qquad \text{and} \qquad S^{*}HS=P,
\end{displaymath} where the  columns of $S$ capture RC basis of $A$ and this basis is Lipschitz stable, that is
\begin{equation*}\label{realperturbation2}
\| S-S_0 \| \leq K(\| A-A_0 \| + \| H -H_0 \|).
\end{equation*}

\subsubsection{Weak Affiliation}


Now let us clarify what the difference is between {\it affiliation} and {\it weak affiliation}.


\begin{definition}{\bf (Weak Affiliation Relation)}\label{waffil}
	Two pairs $(A,H)$ and $(B,G)$ are called {\bf weakly affiliated} if for some invertible matrix $T$ matrices $T^{-1}AT$ and $B$ have the same Jordan basis, and $T^{*}HT=G$.
	The  relation $(A, H) \overset{T}{\mapsto_w} (B,G)$ is called the {\bf weak affiliation relation}.
\end{definition}

In \cite[Theorem 1.8]{BOP} it was showed that if a small perturbation $A$ of $A_0$ has the same Jordan structure as $A_0$, $H$ is a
small perturbation of $H_0$, and as before $A_0$ is $H_0$-selfadjoint and $A$ is $H$-selfadjoint, then  an affiliation
matrix $S$ exists, and it can be chosen to be a small perturbation of the identity matrix, and a
Lipschitz-type bound \begin{equation*}\label{perturbedidentity}
\| I-S \| \leq K(\| A-A_0 \| + \| H -H_0 \|)
\end{equation*}
holds for some positive $K$ depending on $A_0$ only. More details on the weak affiliation case are presented in Section 6.

\subsubsection{The Structure of the Paper}

For the sake of brevity we used examples of FO and FOCS bases to present our main result.   In Section 2 we introduce these concepts rigorously.
        In \cite{DMO} we established the existence of FOCS bases. That construction does not seem to be useful to derive our main result, Lipschitz stability. Hence, in Section 3 we present a different proof for 
        the existence of FOCS bases. The new construction allows us to establish the Lipshitz stability in Section 4. 
       In Section 5 the Lipshitz stability of real canonical forms is established.
       Finally, we slightly generalize the results of Sections 4 and 5 to the weak affiliation case in Section 6.


\section{{Canonical Bases and Forms} }

Let us introduce all the characters involved in this story.
 We have already presented examples of several canonical bases in Section 1. In this section, we formally define FO and CS bases.

\subsection{Flipped-Orthogonal Basis}

This type of bases was presented in \cite{BOP}.
As before, let $J_A=T^{-1}AT$ denote the Jordan canonical form of $A$,
that is 
\begin{displaymath}J_A=
\left [ \begin{array}{cccc}
J(\lambda_1)& {} &{} & 0\\
{} & {}& \ddots  &{}  \\ 
0 & {} &{} & J(\lambda_m) \\
\end{array} \right]\\,
\end{displaymath}
where  each 
\begin{displaymath}J(\lambda_k)=
\left [ \begin{array}{cccc}
\lambda_k & 1& {} &0\\
{} & \ddots & \ddots&  {}\\
{} & {}&\ddots & 1\\
0 & {} &{}&\lambda_k\\
\end{array} \right]\\
\end{displaymath} is a single Jordan block. 
If we partition $T$ correspondingly
\begin{displaymath}
T= [T_1\ | \ldots \ |T_{m} ];
\end{displaymath}
then the columns of $T_k$ form a Jordan chain of $A$ corresponding to $\lambda_k\in\sigma(A)$, and the columns of $T$ form a Jordan basis of $A$.

\begin{definition}{\bf (Affiliation Relation)}\label{affil}
	Two pairs (A,H) and (B,G) are called {\bf affiliated} if $T^{-1}AT=B$ and $T^{*}HT=G$.
	The  relation $(A, H) \overset{T}{\mapsto} (B,G)$ is called the {\bf affiliation relation}.
\end{definition}

Even if  there exists an invertible matrix $T$ such that $(A,H)\overset{T}{\mapsto}(B,G)$, and the matrix $B=T^{-1}AT$ is the Jordan form of $A$, the matrix $G=T^{\ast}HT$ is not necessarily simple.

As before we  consider $H$-selfadjoint matrices $A$. 
From the definition of $H$-selfadjoint matirces we know that $A$ is similar to $A^*$, meaning that nonreal eigenvalues of $A$ come in conjugate pairs. Moreover, the sizes of Jordan blocks of $A$ corresponding to such $\lambda$  and to $\overline{\lambda}$ are identical. Let us denote by $\lambda_1,\ldots, \lambda_\alpha$ all real eigenvalues of $H$-selfadjoint matrix $A$, and by $\lambda_{\alpha+1}, \ldots, \lambda_\beta$ all eigenvalues of $A$ from the upper half plane. Then the Jordan canonical form of $A$ has the form 
\begin{displaymath}T^{-1}AT= J_A=
\left [ \begin{array}{ccccccccc}
J_1& {} &{} & {} &{} &{}& {} & {}  \\
{} & {}& \ddots  &{} & {} &{} &{}& {} \\ 
{} & {}& {}  &J_{\alpha} & {} &{} &{}& {} \\
{}& {} &{} & {} &\hat{J}_{\alpha+1}&{}& {}\\
{} &  {} & {}&{} & {} &\ddots &{}& {}  \\
{} & {}& {}  &{} & {} &{} & {} & \hat{J}_{\beta}\\ 
\end{array} \right]\\
\end{displaymath} where $J_k$  is a single Jordan block corresponding to real $\lambda_k$ as before (for $k=1,\ldots,\alpha$), the matrices 
\begin{displaymath}
\hat{J}_k=\left [ \begin{array}{cc}
J(\lambda_k) &0\\
0 &J(\overline{\lambda}_k)\\
\end{array} \right]
\end{displaymath} for $k=\alpha+1,\ldots,\beta$ are direct sums of two Jordan blocks of the same size corresponding to $\lambda_k$ and $\overline{\lambda}_k$, respectively.

The following fact (see~\cite{GLR83,M63}) shows that there is  such a matrix $T$ that makes $T^*HT$  sparse.

\begin{proposition}\label{thm:main1} Let ${A} \in\mathbb{C}^ {n \times n}$  be a fixed $H$-selfadjoint matrix for some self-adjoint matrix $H$ and $J$ be its Jordan form. Then there exists an invertible matrix $T$ such that 
	\begin{equation}\label{eq:pairs1}
	(A,H)\overset{T}{\mapsto}(J,P)
	\end{equation}
	where
	\begin{equation}\label{J}J= J(\lambda_1)\oplus \cdots J(\lambda_\alpha)\oplus \hat{J}(\lambda_{\alpha+1})\oplus \cdots \oplus\hat{ J}(\lambda_\beta)
	\end{equation}
	and $J(\lambda_i)$ is a  Jordan  block for real eigenvalues $\lambda_1,\ldots,\lambda_{\alpha}$ and 
	\begin{equation}\label{Jk}
	\hat{J}(\lambda_k)=\left [ \begin{array}{cc}
	J(\lambda_k) &0\\
	0 &J(\overline{\lambda}_k)\\
	\end{array} \right]
	\end{equation} is a direct sum of two Jordan blocks of the same size corresponding to $\lambda_k$ and~$\overline{\lambda}_k$,
	\begin{equation}\label{eq:P}
	P=P_1 \oplus\cdots \oplus P_{\alpha} \oplus {P}_{\alpha+1} \oplus \cdots \oplus{P}_{\beta}
	\end{equation}
	where $P_k$ is a sip matrix $\epsilon_k \tilde{I}_k$
	(i.e. $\left[ \begin{smallmatrix}
	&&& \\
	0 & \ldots & 0 &\epsilon_k\\
	\vdots & \invddots & \epsilon_k & 0\\
	0 & \invddots & \invddots & \vdots\\[0.5em]
	\epsilon_{k} & 0 & \cdots & 0\\
	&&& 
	\end{smallmatrix} \right]$) of the same size as $J({\lambda_k})$ and $\epsilon_{k}=\pm{1}$  for $ k=1,\ldots,\alpha$ and a sip matrix $\tilde{I}_k$ (i.e. $\left[ \begin{smallmatrix}
	&&& \\
	0 & \ldots & 0 &1\\
	\vdots & \invddots & 1 & 0\\
	0 & \invddots & \invddots & \vdots\\[0.5em]
	1 & 0 & \cdots & 0\\
	&&& 
	\end{smallmatrix} \right]$) of the same size as  $\hat{J}({\lambda_k})$ for $k=\alpha+1,\ldots,\beta$. 
\end{proposition}

This shows us the existence of a canonical basis for which $P$ has a very particular form.

Think back to the example of $A$ and $H$ from \eqref{A}. An illustration to Proposition \ref{thm:main1} would be matrix $T$ in \eqref{T_1}. The following is a formal definition of the basis captured by columns of $T$.

\begin{definition}{\bf(Flipped-orthogonal Basis)}  Let us partition $T$ in $\eqref{eq:pairs1}$ in the following way
	\begin{displaymath}
	T= [T_1  \ldots T_{\alpha}\ | T_{\alpha+1} \ldots T_{\beta} ].
	\end{displaymath}
	Obviously, for  $k=1,\ldots, \alpha$ the columns of 
		\begin{displaymath}
	T_k=[g_{0,k}, \ldots, g_{{p_k}-1,k}]
	\end{displaymath}
	form the Jordan chain of $A$ corresponding to $\lambda_k$. Similarly, the columns of
		\begin{displaymath}
	T_k=[\underbrace{g_{0,k}, \ldots, g_{{p_k}-1,k}}_{\lambda_k}\ | \underbrace{h_{0,k},\ldots,h_{{p_k}-1,k}}_{\overline{\lambda}_k} ]
	\end{displaymath}
	for $k=\alpha+1,\ldots, \beta$ form two Jordan chains, corresponding to $\lambda_k$ and ${\overline{\lambda}_k}$ respectively. Then the structure of the matrix $P$ in \eqref{eq:P} implies the following orthogonality relation,
	\begin{displaymath}
	g_{ik}^*Hg_{jm} =0, \;\; \;\;
	h_{ik}^*Hg_{jm} =0, \;\; \;\;
	h_{ik}^*Hh_{jm} =0, \;\; \text{ for }\;\; k \neq m.
	\end{displaymath}
	Further, we have that
	\begin{displaymath}
	g_{i,k}^*Hg_{j,k} = \left \{  \begin{array}{ll}
	\epsilon_k, \quad j=p_{k} -1-i\\ 
	0 \quad otherwise
	\end{array} \right.  \qquad for  \quad k=1,\ldots,\alpha
	\end{displaymath}
	\begin{displaymath}
	\left \{  \begin{array}{ll}
	g_{ik}^*Hg_{jk} =0,\\
	h_{ik}^*Hh_{jk} =0,\\ 
	h_{i,k}^*Hg_{j,k} = \left \{  \begin{array}{ll}
	1, \quad j=p_{k} -1-i\\ 
	0 \quad otherwise
	\end{array} \right. 
	\end{array} \right.  \qquad for  \quad k=\alpha+1,\ldots,\beta.
	\end{displaymath}
	Jordan bases, having the above property, are called flipped-orthogonal (FO). 
\end{definition}


\subsection{$\gamma$-Conjugate Symmetric Bases}
If $A \in \mathbb{R}^{ n \times n} $  and $\{q_k \}_{k=0}^{m-1}$ is its Jordan chain corresponding to a nonreal eigenvalue $\lambda$ then  $\{ \gamma \overline{q}_k\}_{k=0}^{m-1}$  is also a Jordan chain of $A$ but corresponding to $\overline{\lambda}$ for any non-zero $\gamma \in \mathbb{C}$. This observation leads to the following definition.

\begin{definition}{\bf ($\gamma$-Conjugate Symmetric Basis)}
	Suppose $A\in \mathbb{R}^{ n \times n} $ and that there exists an invertible matrix $N$ such that $A=NJN^{-1}$, where 
	\begin{equation}\label{hatJ}
	J=J(\lambda_{1})\oplus\cdots \oplus J(\lambda_{\alpha})\oplus \hat{J}(\lambda_{\alpha+1})\oplus \cdots \oplus \hat{J}(\lambda_{\beta})
	\end{equation}is the Jordan canonical form with $\hat{J}(\lambda_k)$ defined in  \eqref{J} and \eqref{Jk}, the matrix
	\begin{equation}{\label{N}}
	N=[N_1 \lvert \ldots \lvert N_{\alpha} \lvert {N}_{\alpha+1} \lvert \ldots \lvert {N}_{\beta}]
	\end{equation}
	can be chosen 
	\begin{equation}\label{eq:t1}
	N_{k}=[Q_{k} \lvert \gamma\overline{Q}_{k}] 
	\end{equation} 
	for some $\gamma\neq0$ and $k=\alpha+1,\ldots,\beta$. The columns of $N$ capture the Jordan basis of $A$, and because of \eqref{eq:t1} we call it the $\gamma$-Conjugate Symmetric Basis.
	
\end{definition}

Now that we have introduced two cannonical bases, FO and $\gamma$-CS, we might ask:
{\it Is there a basis for real $H$-selfadjoint matrices that is both flipped orthogonal and $\gamma$-conjugate symmetric?} The answer to this question is in our next section.


\section{{Existence of a $\gamma$-FOCS Basis} }

The FO basis was defined for pairs of matrices $(A,H),$ where $A$ and  $H$ are not necessarily real. The $\gamma$-CS basis is just defined  for the case of   $A$ being a real matrix.

Proposition {\ref{thm:main1}} implies the existence of the FO basis for $(A,H)$,  the latter, generally, does not have the CS property. Theorem  {\ref{thm:focs1}}  shows that if $(A,H)$ is real, then there exists a $\gamma$-FOCS basis, that is, the one, which is simultaneously  FO and  $\gamma$-CS. 

\begin{theorem}{\bf (Existence of  a $\gamma$-FOCS basis)} \label{thm:focs1} Let all the assumptions of Proposition \ref{thm:main1}  hold true. Assume, in addition, that both matrices $A$ and $H$ are real. Then $N$ in $(A,H)\overset{N}{\mapsto}(J,P)$ can be chosen such that if we partition it in the agreement with \eqref{N} 
	then each $N_{k}$ has the form 
	\begin{equation*}
	N_{k}=[L_{k} \lvert \gamma\overline{L}_{k}] 
	\end{equation*} 
	for some matrices $L_k$ for $k=\alpha+1,\ldots,\beta$.
\end{theorem}

In other words, the columns of $N$ form not only FO, but also a $\gamma$-CS basis. Theorem~\ref{thm:focs1} was proven in \cite{DMO}. Below we need to provide it with a different proof, since the notations and details of this new proof will be used in derivation of our main result, the Lipschitz stability.

To establish that we first need to prove the following lemma.


\begin{lemma} \label{lemma1} Let 
	\begin{displaymath} A=
	\left[ \begin{array}{cccccc}
	J_m{(\lambda)}& 0\\
	0 & J_m{(\overline{\lambda})}
	\end{array}\right],
	\end{displaymath}where $J_m(\lambda)$ is an $m\times m$ Jordan block with an eigenvalue $\lambda \notin \mathbb{R}$. If $A$ is $\widetilde{G}$-selfadjoint, then $\widetilde{G}$ has the following structure
	\begin{equation} \label{lemma:G}\widetilde{G}=
	\left[ \begin{array}{cccccc}
	0& G^{*} \\
	G & 0
	\end{array}\right],\quad where \quad
	G= 
	\left[ \begin{array}{ccccc}
	0 & \ldots & 0 &g_0\\
	\vdots & \invddots & g_0 & g_1\\
	0 & \invddots & \invddots & \vdots\\
	g_0 & g_1 & \cdots & g_{m-1}\\
	\end{array}\right],
	\end{equation} with some coefficients $g_j$'s.
\end{lemma}

\begin{proof}
	Recall that $\{ e_k \}$ form a Jordan chain of $A$ means 
	\begin{equation}\label{eq:chain1}
	(A-\lambda I)e_{1}=0, \qquad (A-\lambda I){e}_{k+1}={e}_{k} \qquad k=1, \ldots, m-1
	\end{equation}
	and 
	\begin{equation} \label{eq:chain2}
	(A- \overline{\lambda} I)e_{m+1}=0, \qquad (A-\overline{\lambda} I){e}_{k+1}= {e}_k \qquad k=m+1, \ldots, 2m-1.
	\end{equation}
	Let 
	\begin{displaymath}\widetilde{G}= 
	\left[ \begin{array}{ccc}
	X &Y\\
	Z&U\\
	\end{array}\right],
	\end{displaymath}
	where $X=\{[e_i,e_j]\}_{i,j=1}^m$,  $U=\{[{e}_i,e_j]\}_{i,j=m+1}^{2m}$,  $Z=\{[e_i,{e}_j]\}_{i=1,j=m+1}^{m,2m}$ and $Y=Z^*$ by conjugate symmetry of indefinite inner products.
	
	Note that  $X$ is the zero matrix.
	Using \eqref{eq:chain1}, we have
	\begin{displaymath}
	\lambda[e_1,e_1]=[\lambda e_1,e_1]=[Ae_1,e_1]=[e_1, Ae_1]=[e_1,\lambda e_1]
		\end{displaymath}	\begin{displaymath}
	=\overline{\lambda}[e_1,e_1]\Rightarrow [e_1,e_1]=0, 
	\end{displaymath} 
	since $\lambda\neq\overline{\lambda}$. 
	Moreover,  for any $k=1,\dots,m-1$
	\begin{equation*}
	(\overline{\lambda}-{\lambda})[e_{k+1},e_1]=[e_{k+1},(\lambda-\overline{\lambda})e_1]=[e_{k+1}, (A-\overline{\lambda} I)e_{1}]
	\end{equation*}
	\begin{equation}\label{1col}
	= [(A-\lambda I)e_{k+1},e_{1}]=[e_k,e_{1}] 
	\end{equation}
	and
	\begin{equation*}
	({\lambda}-\overline{\lambda})[e_{1},e_{k+1}]=[(\lambda-\overline{\lambda})e_{1},e_{k+1}]= [(A-\overline{\lambda} I)e_{1},e_{k+1}]
	\end{equation*}
	\begin{equation}\label{1row}
	=[e_{1}, (A-{\lambda} I)e_{k+1}]=[e_1,e_{k}].
	\end{equation}
	Keeping in mind that $ \lambda\neq\overline{\lambda}$ and $[e_1,e_1]=0$, we get $[e_1,e_{k}]=[e_k,e_{1}]=0$ by induction from \eqref{1col} and \eqref{1row} for $k=1,\dots,m$.
	
	Suppose $[e_{i-1},e_j]=0$ and
	$[e_i, e_{j-1}]=0$ for some $i,j=2,\dots,m$.
	Next, 
	\begin{displaymath}
	\lambda[e_{i},e_j]=[\lambda e_i,e_j]=[Ae_{i}-e_{i-1},e_j]=[Ae_{i},e_j]-{[e_{i-1},e_j]}=[e_i,Ae_j]=
	\end{displaymath}
	\begin{displaymath}
	=[e_i,\lambda e_j +e_{j-1}]=[e_i,\lambda e_{j}]-{[e_{i},e_{j-1}]} =\overline{\lambda}[e_i,e_j]\quad\text{ or }\quad [e_i,e_j]=0,
	\end{displaymath}
	since $ \lambda\neq\overline{\lambda}$. Thus, by induction on $i$ and $j$ we have $X=0$. 
	
	The fact that $U=0$ can be proved using similar argument.
	
Next, let us see what happens to
	\begin{equation}\label{lemma:Z}Z= 
	\left[ \begin{array}{cccc}
	\ [e_{1},{e}_{m+1}] & [e_{1},{e}_{m+2}]  & \ldots & [e_{1},e_{2m}] \\
	\ [ e_{2},e_{m+1}] & [e_{2},e_{m+2}] &\ldots & [e_{2},e_{2m}]\\
	\vdots & \vdots & \ddots & \vdots \\
	\ [e_{m},e_{m+1}] &[e_{m},e_{m+2}] & \ldots & [e_{m},e_{2m}] \\
	\end{array} \right].
	\end{equation}
	
Applying \eqref{eq:chain2}	 for  $k=1,2, \ldots, m-1$ we get
	\begin{equation}\label{lemma:upper}
	[e_k,e_{m+1}]= [(A-\lambda I)e_{k+1},e_{m+1}]=[e_{k+1}, (A-\overline{\lambda} I)e_{m+1}]=[e_{k+1},0]=0
	\end{equation}
	\begin{equation}\label{lemma:upper1}
	[e_m,e_{m+1}]=[e_{m}, (A-\overline{\lambda} I)e_{m+2}]= [(A-\lambda I)e_{m},e_{m+2}]=[e_{m-1},e_{m+2}].
	\end{equation}
This means that  matrix $Z$ has the Hankel structure as shown in \eqref{lemma:G}.  We have
	\begin{equation}\label{lemma:hankel}
	[e_k,e_{j}]= [(A-\lambda I)e_{k+1},e_{j}]=[e_{k+1}, (A-\overline{\lambda} I)e_{j}]=[e_{k+1},e_{j-1}]
	\end{equation}
	for  $ k=1,\ldots, m-1$ and $j=m+2,\ldots, 2m$.
	The matrix in \eqref{lemma:Z} is Hankel and becomes $G$ from \eqref{lemma:G}  by using \eqref{lemma:upper}--\eqref{lemma:hankel} and by letting $[e_k,e_{2m}]=g_{k-1}$. 
\end{proof}

We are at the point now to prove the main result of this section. We are going to construct the columns of matrix $N$ in Theorem~\ref{thm:focs1} in a way that it captures the~$\gamma$-FOCS  basis of $(A,H)$. This matrix $N=Z_1Z_2Z_3Z_4$ is constructed in four steps, as described in  Table {\ref{table:constructing N}} below.
\begin{table}[h]
	\centering
	\begin{tabular}{ | m{4.3cm}  |m{7.3cm} | } 
		\hline 
		Constructing $Z_1,Z_2,Z_3,Z_4$ & $(A,H)\overset{Z_1}{\mapsto} (J,\tilde{G}) \overset{Z_2}{\mapsto} (J, \tilde{G}_1) \overset{Z_3}{\mapsto} (J,\tilde{G}_2) \overset{Z_4}{\mapsto} (J,P)$\\[1ex]
		\hline
		Step 1 (constructing $Z_1$) &$ (A,H)\overset{Z_1}{\mapsto} (J,\tilde{G})$\\[1ex]
		Step 2 (constructing $Z_2$) &$(A,H)\mapsto\overset{Z_1}{\mapsto}\overset{Z_2}{\mapsto}\mapsto{}(J,\tilde{G}_1)$\\[1ex]
		Step  3  (constructing $Z_3$) &$(A,H)\mapsto\overset{Z_1}{\mapsto}\overset{Z_2}{\mapsto}\overset{Z_3}{\mapsto}\mapsto (J,\tilde{G}_2)$\\[1ex]
		Step  4  (constructing  $Z_4$)& $(A,H)\mapsto\overset{Z_1}{\mapsto}\overset{Z_2}{\mapsto}\overset{Z_3}{\mapsto}\overset{Z_3}{\mapsto}\overset{=N}{\mapsto}\mapsto (J,P)$\\ [1ex]
		\hline
	\end{tabular}
	\caption{Constructing N}
	\label{table:constructing N}
\end{table}

We will construct it in a way that the columns of  each of four matrices $Z_1$, $Z_1Z_2$, $Z_1Z_2Z_3$, $Z_1Z_2Z_3Z_4$ capture the  $\gamma$-CS bases of $A$, and  columns  of the last  matrix $N=Z_1Z_2Z_3Z_4$ capture not only  $\gamma$-CS, but also the desired $\gamma-$FOCS basis.

\begin{proof}[Proof of the Theorem {\ref{thm:focs1}}] Let us first introduce the following notation that we will use throughout this proof. We denote by $I_k$ (resp. $\widetilde{I}_k$) the identity (resp. sip) matrix of the size of $J(\lambda_k)$ from \eqref{hatJ}.
	
Now let us go through the procedure step-by-step:
	
{\bf Step 1 (Mapping $A$ $\rightarrow$ $J$. Constructing $Z_1$)} We know from Proposition  \ref{thm:main1} that there exists $T$ whose columns capture the FO-basis of $(A,H)$, that is $(A,H) \overset{T}{\mapsto} (J,P)$. Let us partition
\begin{equation*}\label{T_1}
T=[T_1 \lvert \ldots \lvert T_{\alpha} \lvert T_{\alpha+1} \lvert \ldots \lvert T_{\beta}]
\end{equation*}
in correspondence with \eqref{hatJ}. The columns of $T$ do not necessarily form a  $\gamma$-CS basis yet,  that is
\begin{equation*}
T_k=[L_{k} \lvert S_{k}] \quad\text{ for }\quad k=\alpha+1, \ldots, \beta
\end{equation*} 
with some  matrices $L_k$ and $S_k$ possibly not to related to each other.
Now, define
\begin{equation}\label{def:T_1}
Z_1=[T_1 \lvert \ldots \lvert T_{\alpha} \lvert {M}_{\alpha+1} \lvert \ldots \lvert {M}_{\beta}],
\end{equation}
where
\begin{equation}\label{def:T_k}
M_{k}=[L_{k} \lvert \gamma \overline{L}_{k}]  \quad\text{ for }\quad k=\alpha+1, \ldots, \beta.
\end{equation} Thus, we got $Z_1$ for  $(A,H)\overset {Z_1}{\mapsto} (J,\widetilde{G})$. Moreover, by using the same partition of $\widetilde{G}$ as in  \eqref{eq:P}, for each block corresponding to non-real eigenvalues we observed the form as in \eqref{lemma:G} for $k=\alpha+1,\dots,\alpha+\beta$
\begin{equation*} \widetilde{G_k}=
\left[ \begin{array}{cccccc}
0& G_k^{*} \\
G_k & 0
\end{array}\right],\quad\text{ where }\quad
G_k= 
\left[ \begin{array}{ccccc}
0 & \ldots & 0 &g_0^{(k)}\\
\vdots & \invddots & g_0 ^{(k)}& g_1^{(k)}\\
0 & \invddots & \invddots & \vdots\\
g_0^{(k)} & g_1^{(k)} & \cdots & g_{k-1}^{(k)}\\
\end{array}\right]
\end{equation*} 
and $\widetilde{G}_k$ is the sip matrix $\epsilon_k\widetilde{I}_k$ ($\epsilon=\pm1$)  for $k=1,\ldots,\alpha$.

By construction the columns of matrix $Z_1$ capture the  $\gamma$-CS basis of $A$. However, it is not necessarily the FO basis yet.


{\bf Step 2 (Forcing the antidiagonal of $\widetilde{G}_1$ to be real. Constructing $Z_2$)} During this step, we use the argument and modulus of $g_0^{(k)}=  e^{i\varphi_k}r_k$  for $k=\alpha+1,\dots,\alpha+\beta$ to define \begin{equation*}
Z_2=I_1 \oplus \ldots \oplus I_{\alpha} \oplus Z^{(2)}_{\alpha+1} \oplus \ldots \oplus  Z^{(2)}_{\beta},
\end{equation*}
where
\begin{equation}\label{Z2}
Z_k^{(2)}=\left[ \begin{array}{cc}
e^{\frac{-{i\varphi_k}}{2}}\sqrt{\frac{s_k}{r_k}}I & 0\\
0 & e^{\frac{{i\varphi_k}}{2}}\sqrt{\frac{s_k}{r_k}}I \\
\end{array}\right],
\end{equation}
with $s_k=|\text{Re}\,g_0^{(k)}|$.
Then,
\begin{displaymath} {Z_2}^{*}\widetilde{G}{Z_2}=
\widetilde{G_1}, \qquad {Z_2}^{-1}J{Z_2}=J,
\end{displaymath} 
and the corresponding to  \eqref{hatJ} partition of $\widetilde{G_1}$ is such that $\widetilde{G_k}^{(1)}$'s are unchanged for $ k=1,\ldots,\alpha$ and $\widetilde{G_k}^{(1)}=\left[ \begin{array}{cccccc}
0&\left(G_k^{(1)}\right)^* \\
G_k^{(1)}& 0
\end{array}\right]$
where
${G_k^{(1)}}= e^{-{i\varphi_k}}{\frac{s_k}{r_k}}G_k=\left[ \begin{array}{ccccc}
0 & \ldots & 0 &s_k\\
\vdots & \invddots & s_k & *\\
0 & \invddots & \invddots & \vdots\\
s_k & * & \cdots & *\\
\end{array}\right]$ for $k=\alpha+1,\dots,\alpha+\beta$.

If we put $Z_1$ and $Z_2$ together, i.e. 
\begin{displaymath}
Z_1Z_2=[T_1\lvert \ldots\lvert  T_{\alpha}\lvert M_{\alpha+1}^{(0)} \lvert\dots\lvert M_{\beta}^{(0)} ].
\end{displaymath}
we get that
\begin{equation}\label{Z1Z2}
M_k^{(0)}= [e^{\frac{i\varphi_k}{2}} \sqrt{\frac{s_k}{r_k}} L_k  \lvert  \gamma e^{\frac{-i\varphi_k}{2}} \sqrt{\frac{s_k}{r_k}} \overline{L}_k] =: [Q_k^{(0)} \lvert \gamma \overline{Q}_k^{(0)}]\quad\text{ for }\quad k=\alpha+1, \ldots, \beta,
\end{equation} where $Q_k^{(0)}=e^{\frac{i\varphi_k}{2}} {\sqrt{\frac{s_k}{r_k}}} L_k$. As we see from \eqref{Z1Z2},  the columns $Z_1Z_2$ still capture the  $\gamma$-CS basis of $A$.


{\bf Step 3 (Making $G_1$ the unit-antidiagonal. Constructing $Z_3$)} This time we just scale all the columns. We use the result from Step 2 to define 
\begin{equation*}
Z_3=I_1 \oplus  \ldots \oplus  I_{\alpha} \oplus  Z^{(3)}_{\alpha+1} \oplus  \ldots \oplus  {Z}_{\beta}^{(3)},
\end{equation*}
where
$
Z_k^{(3)}=\frac{1}{\sqrt{s_k}}I
$  for $k=\alpha+1,\dots,\alpha+\beta$. That is we simply scaling $\widetilde{G}_1$.
Then
\begin{displaymath}Z_3^{*}\widetilde{G}_1Z_3= \widetilde{G}_2, \qquad Z_{3}^{-1}{J}Z_3={J},
\end{displaymath} 
where as before we have the block structure of $\widetilde{G_2}$ such that  $\widetilde{G_k}^{(2)}$  being a sip matrix  of the same size as $J({\lambda_k})$ for $ k=1,\ldots,\alpha$ and $\widetilde{G_k}^{(2)}=\left[ \begin{array}{cccccc}
0&\left(G_k^{(2)}\right)^{*} \\
G_k^{(2)}& 0
\end{array}\right]$ with 
$
{G}_k^{(2)}=
\left[ \begin{array}{ccccc}
0 & \ldots & 0 &1\\
\vdots & \invddots &1&* \\
0 & \invddots  &  \invddots & \vdots\\
1 & *& \dots &*\\
\end{array}\right]$ for $k=\alpha+1,\dots,\alpha+\beta$.

We constructed such  matrix $Z_3$ in order to have
\begin{displaymath}
Z_1Z_2Z_3=[T_1\lvert \ldots\lvert T_{\alpha}\lvert 
M_{\alpha+1}^{(1)} \lvert \dots \lvert
M_{\beta}^{(1)}] .
\end{displaymath} 
where we get the following form for $M_k^{(1)}$'s
\begin{displaymath}
M_{k}^{(1)}= [ \frac{1}{\sqrt{s}} Q_k^{(0)} \lvert \gamma  \frac{1}{\sqrt{s}} \overline{Q}_k^{(0)}]=: [Q_k^{(1)} \lvert \gamma \overline{Q}_k^{(1)}],
\end{displaymath}
by letting $Q_k^{(1)}= \frac{e^{\frac{i\varphi_k}{2}}}{\sqrt{r_k}} L_k$. Then, we have $(J,\widetilde{G}_1) \overset{Z_{3}}{\mapsto}(J,\widetilde{G_2})$. Moreover,  $(A,H) \overset{Z_1Z_{2}Z_3}{\mapsto} (J,\widetilde{G_2})$. Therefore, the columns of $Z_1Z_2Z_3$ capture the~$\gamma$-CS basis of $A$ but as before we cannot guarantee it to be FO yet. The flipped orthogonality property is still missing.

\vskip 10pt


{\bf Step 4 (Zeroing sub-antidiagonal entires of  $G_2$). Constructing $Z_4$} Now, we look for a matrix $Z_4$ such that $N=Z_1Z_2Z_3Z_4$ has all the properties formulated in the theorem. Let us define  $Z_4$ in the following way 
\begin{displaymath}
Z_4= I_1\oplus  \ldots\oplus  I_\alpha\oplus  Z_{\alpha+1}^{(4)}\oplus  \ldots\oplus  Z_{\alpha+ \beta}^{(4)},\quad 
\end{displaymath} where $ Z_k^{(4)}=
\left[ \begin{array}{ccc}
	F_k &0 \\
	0 & \overline{F}_k
\end{array}\right]$ for $k=\alpha+1,\dots,\alpha+\beta$. 
The exact formulas for $F_k$'s will be found later. Then, once again, we get that
\begin{displaymath}
Z_1Z_2Z_3Z_4=[T_1\lvert \ldots\lvert T_{\alpha}\lvert 
M_{\alpha+1}^{(2)} \lvert \dots \lvert
M_{\beta}^{(2)}],
\end{displaymath}where
\begin{equation}{\label{n}}
M_{k}^{(2)}= [Q_k^{(1)} \lvert \gamma \overline{Q}_k^{(1)}]\left[ \begin{array}{cccccc}
F_k &0 \\
0 & \overline{F}_k
\end{array}\right]= [Q_k^{(1)}F_k \lvert \gamma \overline{Q}_k^{(1)}\overline{F}_k]=[Q_k^{(2)} \lvert \gamma\overline{Q}_k^{(2)}]
\end{equation}
with $Q_k^{(2)}=e^{\frac{i\varphi_k}{2}} \frac{1}{\sqrt{r_k}} L_kF_k$ will have the desired $\gamma$-CS form with any invertible $F_k$. In addition, we want  $F_k$ to  satisfy 
\begin{equation}\label{eq:t_3} \quad \left(Z_k^{(4)}\right)^{-1}
\left[\begin{array}{cc}
J(\lambda_k)& 0\\
0& {J(\overline{\lambda}_k)}
\end{array}\right]Z_k^{(4)}=
\left[\begin{array}{cc}
J(\lambda_k)& 0\\
0& {J(\overline{\lambda}_k)}
\end{array}\right].\end{equation}
The  relation in \eqref{eq:t_3} is satisfied if $F_k$ is an upper triangular Toeplitz matrix that commutes with $J(\lambda_k)$.
Furthermore, an immediate computation shows that 
\begin{displaymath}Z_k^{(4)*}\widetilde{G_k}^{(2)}Z_k^{(4)}=
\left[ \begin{array}{cccccc}
F_{k}^*& 0 \\
0& F_k^{T}
\end{array}\right]
\left[\begin{array}{cc}
0&\left(G_{k}^{(2)}\right)^*\\
G_k^{(2)}&0
\end{array}\right]
\left[\begin{array}{cc}
F_k&0\\
0&\overline{F}_k
\end{array}\right]
\end{displaymath}
\begin{equation*}
=
\left[\begin{array}{cc}
0&F_k^{*}G_k^{(2)*}\overline{F}_k\\
F_k^TG_k^{(2)}F_k& 0
\end{array}\right]
=
\left[\begin{array}{cc}
0&(F_k^TG_k^{(2)}F_k)^*\\
F_k^TG_k^{(2)}F_k&0
\end{array}\right].
\end{equation*} 
 
Hence, if $F_k$ satisfies
\begin{equation}\label{eq:f} 
F_k^TG_k^{(2)}F_k=\widetilde{I}_k
\end{equation} 
for $k=\alpha+1,\dots,\alpha+\beta$, then
\begin{equation*}\label{eq:t_3 1}
Z_4^{*}\widetilde{G_2}Z_4=P
\end{equation*}
and we are done.

Therefore, it remains to solve \eqref{eq:f} for an upper triangular Toeplitz matrix $F_k$, and the proof will be complete. Since $\widetilde{I}^2=I$, we have 
\begin{displaymath} F_k^TG_k^{(2)}\widetilde{I}_k\widetilde{I}_kF_k\widetilde{I}_k= \widetilde{I}_k \widetilde{I}_k=I_k,\hspace{.5cm} \text{ then } \hspace{.5cm} \underbrace{F_k^\top}_{F}\underbrace{G_k^{(2)}\widetilde{I}_k}_{G_3}\underbrace{\widetilde{I}_kF_k\widetilde{I}_k}_{F}= I_k.
\end{displaymath}
Since the lower triangular Toeplitz matrices $F$ and $G_3$ commute,  we have
\begin{displaymath}
F^2G_3=I_k. 
\end{displaymath}
We see that $F^2$ is the inverse of the unit lower triangular Toeplitz matrix  $G_3$. Hence, $F$ can be chosen to be a unit lower triangular Toeplitz matrix. Thus, $F_k=F^\top$ is a unit upper triangular Toeplitz matrix and, hence, it commutes with $J(\lambda_k)$.


{\bf Step 5 (Combining the affiliation matrices $Z_1,Z_2, Z_3$, and $Z_4$)} By combining all the previous steps, we get the desired conclusion \begin{displaymath}
(A,H) \overset{N}{\mapsto} (J,P)
\end{displaymath}where $N=Z_1Z_2Z_3Z_4$. In other words, the columns of $N$ form a flipped-orthogo\-nal Jordan basis. Furthermore, \eqref{n} shows that this Jordan basis is also $\gamma$-conjugate-symmetric. This concludes the proof of Theorem {\ref{thm:focs1}}.
\end{proof}


\section{Stability of $\gamma$-FOCS Canonical Forms and Bases}

Now, we are ready to state  the main result of this section. 

\begin{theorem}{\bf(Lipschitz stability of $\gamma$-FOCS basis)} \label{thm:main2} Let $ A_0 \in \mathbb{R}^{ n\times n}$ be a fixed $H_0$-selfadjoint matrix, where $H_0\in \mathbb{R}^{n \times n}$,  and let the columns of this matrix capture the FOCS-basis, that is, 
	\begin{displaymath}
	(A_0,H_0)\overset{T_0}{\mapsto}(J,P).
	\end{displaymath}
	There exist constants $K,\delta > 0$ (depending on  $A_{0}$ and $ H_0 $ only)  such that  the following assertion holds for any $H$-selfadjoint matrix $A$ with the pair $(A,H)$ real and $A$ invertible and similar to $A_0$, and
	
	\begin{equation}\label{delta}
	\| A-A_0\| + \| H-H_0\| < \delta,
	\end{equation}
	there is an invertible matrix $N \in \mathbb{R}^{n \times n}$, whose columns capture the $\gamma$-FOCS basis of $A$, i.e. 
	\begin{displaymath}
	(A,H)\overset{N}{\mapsto}(J,P),
	\end{displaymath}and such that 
	\begin{equation}\label{NN0}
	\| N - T_{0} \| \leq K(\| A-A_0 \| + \| H-H_0 \|).
	\end{equation} 
\end{theorem}

This theorem says that $\gamma$-FOCS bases are Lipshitz stable under small perturbations, preserving the Jordan structure of $A_0$. To prove Theorem {\ref{thm:main2}} we need a technical result stated in Lemma \ref{lemma2} and in this proof we need  to  modify the notations as described next.

\begin{displaymath}
\bf{\Theta-Calculus}
\end{displaymath} In this paper we use and prove a large number of inequalities with the right-hand side 
\begin{displaymath}
K(\| A-A_0\| + \| H-H_0 \|)
\end{displaymath}
with some constant $K>0$. Each of these inequalities hold for its own $K$ implying cumbersome notations (introducing large number of $K_i$'s and the necessity to explain at each step that those $K_i$'s are different). To avoid it we introduce the notation
\begin{displaymath}
\Theta =K(\| A-A_0\| + \| H-H_0 \|),
\end{displaymath}where we do not keep track of the actual value of $K$. Though it is not quite rigorous (e.g, $3\Theta\leq 2\Theta$ is a valid inequality) but this does not lead to a confusion, and greatly simplifies the presentation. However, one has to always keep in mind that in different inequalities $\Theta$'s may invoke different values of $K$.


\begin{lemma}\label{lemma2} Let $T_0$ and $N$ be given matrices such that  $N=Z_1Z_2 Z_3Z_4$ with some $Z_1$, $Z_2$, $Z_3$, $Z_4$. For any $H$-selfadjoint matrix $A$, having the same Jordan structure as $A_0$, if 
	\begin{equation}\label{boundt0}
	\| Z_1-  T_0 \| <  \Theta,
	\end{equation}
	\begin{equation}\label{boundtk}
	\| Z_k-  I \| < \Theta \quad\text{ for } \quad k=2,3,4,
	\end{equation} 
	then we have
	\begin{equation*}
	\| N-T_0\|   <   \Theta.
	\end{equation*}
	
\end{lemma}
\begin{proof}
	Suppose that \eqref{boundt0} and  \eqref{boundtk} are true.  Obviously, the chain of inequalities
	\begin{equation*}
	\| Z_1 \|= \| Z_1-T_0 +T_0 \| \leq \| Z_1 -T_0\| + \| T_0 \|  < \Theta+ \| T_0\| 
	\end{equation*}holds true due to \eqref{boundt0}.
	Now, using inequalities $\eqref{boundt0} \  \text{and} \ \eqref{boundtk} $ again, we get
	\begin{displaymath}
	\| N-T_0\| = \| Z_1Z_2Z_3Z_4-T_0 \|= \| Z_1Z_2Z_3Z_4 - Z_1Z_2Z_3 + Z_1Z_2Z_3 -Z_1Z_2 +Z_1Z_2  
	\end{displaymath}
	\begin{displaymath}
	-Z_1+Z_1 -T_0 \|= \| Z_1Z_2Z_3(Z_4 -  I) \| + \| Z_1Z_2(Z_3-I ) \|  + \| Z_1(Z_2-I) \| 
	\end{displaymath}
	\begin{displaymath}
	+ \| Z_1-T_0 \| 
	\leq \| Z_1Z_2Z_3  \|  \| Z_4 - I \|+ \| Z_1Z_2 \| \| Z_3 -I \| +\| Z_1\| \| Z_2-I\| 
	\end{displaymath}
	\begin{displaymath}
	+  \| Z_1-T_0 \| = \| Z_1Z_2Z_3 -Z_1Z_2 +Z_1Z_2-Z_1+Z_1 \| \| Z_4- I \|+ \| Z_1Z_2-Z_1
	\end{displaymath}
	\begin{equation*}
	+Z_1 \| \| Z_3 -I \| + \| Z_1\|\| Z_2-I\|+  \| Z_1-T_0 \|\leq \big(\| Z_1Z_2(Z_3 - I) \|
	\end{equation*}
	\begin{displaymath}
	+\| Z_1(Z_2-I)\|+\| Z_1\| \big) \| Z_4 - I \|
	+\big ( \| Z_1(Z_2-I) \|+ \| Z_1\| \big)\| Z_3 -I \|
	\end{displaymath}
	\begin{displaymath}
	+\| Z_1\| \| Z_2-I\|
	+\| Z_1-T_0 \|  
	<  \big ( \| Z_1Z_2 \| \| Z_3 - I \| + \| Z_1\| \| Z_2- I \|
	\end{displaymath}
	\begin{equation*}
	+\| Z_1 \| \big ) \Theta + \big ( \| Z_1 \|  \| Z_2 - I \|+ \| Z_1 \|  \big )  \Theta +  \| Z_1 \| \Theta + \Theta \quad 
	\end{equation*}
	\begin{displaymath}
	< \big (  \| Z_1Z_2 -Z_1 +Z_1 \| \Theta + \| Z_1 \| \Theta + \| Z_1 \|  \big ) \Theta + \big( \| Z_1 \| \Theta+\| Z_1 \| \big) \Theta
	\end{displaymath}
	\begin{displaymath}
	+ \| Z_1\| \Theta + \Theta=\big( \| Z_1(Z_2 -I) +Z_1 \| \Theta  + \| Z_1\| \Theta+ \| Z_1 \| \big) \Theta
	\end{displaymath}
	\begin{displaymath}
	+ \big( \| Z_1 \| \Theta+\| Z_1 \| \big) \Theta + \| Z_1\| \Theta + \Theta< \big( ( \| Z_1\| \Theta +\| Z_1 \| ) \Theta  + \| Z_1\| \Theta
	\end{displaymath}
	\begin{displaymath}
	+ \| Z_1 \| \big) \Theta + \big( \| Z_1 \| \Theta+\| Z_1 \| \big) \Theta + \| Z_1\| \Theta + \Theta	\end{displaymath}
	\begin{equation*}< \Theta \big( ( \| T_0 \| + \Theta)
	\cdot ({\Theta}^2 +3 \Theta+1)+1 \big) \leq \Theta.  
	\end{equation*}
\end{proof}

Now we are in a position to prove  Theorem {\ref{thm:main2}}.


{\bf Plan of the proof  of Theorem {\ref{thm:main2}}:} In fact, the desired matrix $N$ satisfying \eqref{NN0} has already been designed in the proof of Theorem \ref{thm:focs1}. It was  constructed there in four steps with matrix $Z_k$ chosen in the $k$th step (see Table~\ref{table:constructing N}). $N$ was given by  $N=Z_1Z_2Z_3Z_4$. Here, we repeat these steps, preserving all the notations from the proof of Theorem \ref{thm:focs1}, and prove the inequalities \eqref{boundt0} and \eqref{boundtk} for each step. In view of Lemma \ref{lemma2} this will completely prove  Theorem \ref{thm:main2}.
\begin{proof}[Proof  of Theorem {\ref{thm:main2}}]
We are ready to proceed with the 1st step of this plan.

{\bf Stability of Step 1.
} Again, let us consider
$(A, H)\overset{T}{\mapsto}(J, P)$ { and } $(A_0, H_0)\overset{T_0}{\mapsto}(J, P)$.
Proposition \ref{thm:stabilityFO} yields that $T$ can be chosen such that 
\begin{equation}\label{TT0theta}
\| T-T_0 \| < \Theta.
\end{equation}Let us partition $T_0$ and $T$ in correspondence with \eqref{hatJ}
\begin{displaymath}
T=[T_1 \lvert \ldots \lvert T_{\alpha} \lvert T_{\alpha+1} \lvert \ldots \lvert T_{\beta}],
\end{displaymath} 
\begin{displaymath}
T_0= [T_1^{(0)} \lvert \ldots \lvert T_{\alpha}^{(0)} \lvert T_{\alpha+1}^{(0)} \lvert \ldots \lvert T_{\beta}^{(0)}].
\end{displaymath} Recall that $T_0$ captures the $\gamma$-FOCS basis of $A_0$, while $T$, according to Proposition~{\ref{thm:stabilityFO}},  captures only the FO-basis not necessarily $\gamma$-CS. Hence,
\begin{equation*}
T_{k}^{(0)}=[L_{k}^{(0)} \lvert \gamma \overline{L_{k}}^{(0)} ], \quad \text{ while }\quad T_k=[L_{k} \lvert S_{k}] \ \text{ for }\quad k=\alpha+1, \ldots, \beta
\end{equation*} 
with some possibly not related to each other matrices $L_k$ and $S_k$. 

According to Proposition \ref{thm:stabilityFO}
\begin{displaymath}
\| L_k - L_k^{(0)} \| \leq \| {T} -T_0 \|\text{ and }\| S_k - \gamma \overline{L_k}^{(0)} \| \leq \| {T} -T_0 \|.
\end{displaymath}
The former and \eqref{TT0theta} imply that the matrix $Z_1$ defined in \eqref{def:T_1} satisfies 
\begin{equation}\label{perturbationZ1T0}
\| Z_1 - T_{0} \| \leq\Theta.
\end{equation} 

Note that 
$$
    \|\tilde{G}-P\|=\|Z_1^*HZ_1-T_0^*H_0T_0\|\le\|Z_1-T_0\|\|H\|\|Z_1\|+\|T_0\|\|H-H_0\|\|Z_1\|+$$
    \begin{equation}\label{GP}
+\|T_0\|\|H_0\|\|Z_1-T_0\|\le\Theta.
\end{equation}
The first inequality in Lemma \ref{lemma2} which is \eqref{boundt0} holds true. Let us check the rest.


{\bf Stability of Step 2.} Our goal is to prove the following  inequality for $Z_2$ defined in \eqref{Z2}
\begin{equation}\label{boundZ_2}
\| I-Z_2 \|  \leq \Theta.
\end{equation}
From the way we defined $Z_2$, we got 
\begin{displaymath}
\| I - Z_2 \| \leq \underset{k}{\max} \Big \{  \Big \lvert 1- e^{-\frac{i\varphi_k}{2} }\sqrt{\frac{s_k}{r_k}}  \Big \lvert, \Big \lvert 1- e^{\frac{i\varphi_k}{2} }\sqrt{\frac{s_k}{r_k}}  \Big \lvert  \Big \}.
\end{displaymath}
The next part is analogous to the argument in \cite[Lemma 4.5]{BOP}. We have
\begin{equation*}
\left| 1- e^{-\frac{i\varphi_k}{2} }\sqrt{\frac{s_k}{r_k}}  \right|= \left| \frac {\sqrt{r_ke^{i\varphi_k}} - \sqrt{s_k}}{\sqrt{r_ke^{i\varphi_k}}}\right| =\left|  \frac {\sqrt{g_0^{(k)}} -\sqrt{s_k}}{\sqrt{g_0^{(k)}}}\right| =\left|\frac {g_0^{(k)} -s_k}{g_0^{(k)}+\sqrt{g_0^{(k)}s_k}}\right|
\end{equation*}
\begin{equation}\label{1}
\leq \frac {\vert  g_0^{(k) }-1 \vert }{\vert g_0^{(k)} \vert }\leq \frac{1}{\vert g_0^{(k)} \vert} \| \tilde{G} -P \| \leq \Theta.
\end{equation}
The last inequality in \eqref{1} follows from  \eqref{GP}.
Therefore, we see that the desired inequality for $Z_2$ \eqref{boundZ_2} is now proven.


{\bf Stability of Step 3.}  During this step, we need to  prove the following inequality
\begin{equation}\label{boundZ_3}
\| I-Z_3 \|  \leq \Theta.
\end{equation} We constructed $Z_3$ so
\begin{displaymath}
\| I - Z_3 \| \leq \underset{k}{\max} \|  I_k -\frac{1}{\sqrt{s_k}}I_k \| \leq  \underset{k}{\max} \Big \lvert 1 -\frac{1}{\sqrt{s_k}}  \Big \lvert.
\end{displaymath}
Note that $s_k$ is a real number. Hence we get
\begin{displaymath}
\Big \lvert 1- \frac{1}{\sqrt{s_k}} \Big \lvert = \Big \lvert  \frac {\sqrt{s_k } - 1}{\sqrt{s_k }}\Big \lvert =\Big \lvert  \frac {s_k -1}{\sqrt{s_k }(\sqrt{s_k }+1)}\Big \lvert =\Big \lvert  \frac {s_k -1}{s_k +\sqrt{s_k }}\Big \lvert \leq \frac {\vert  s_k -1 \vert }{\vert s_k  \vert } 
\end{displaymath}
\begin{equation*}
\leq \frac{\vert g_0^{(k)} -1 \vert}{\vert s_k  \vert} \leq \frac {\| \tilde{G} -P \|}{\vert s_k \vert} \leq \Theta.
\end{equation*}
Thus, \eqref{boundZ_3} holds true.


{\bf Stability of Step 4.} To prove the bound on $\| I - Z_4 \| $, we observe first that
\begin{equation*}
\| \tilde{G}_1 -P \|= \| Z_2\tilde{G}Z_2 -P \| \leq
\| Z_2^{*}\tilde{G}Z_2 - Z_2^{*}PZ_2+Z_2^{*}PZ_2 - Z_2^{*}P 
\end{equation*}
\begin{displaymath}
+Z_2^{*}P -P\| \leq \|  Z_2^{*}(\tilde{G}-P)Z_2 \| + \| Z_2^{*}P(Z_2-I) \| + \| Z_2- I \| \| P \|
\end{displaymath}
\begin{equation*}
\leq \| Z_2 \|^2 \| \tilde{G}-P \| +( \| Z_2\| +1)\| P \|  \| Z_2-I \|
\end{equation*}
\begin{equation}\label{perturbationGP}
< \big( (1+ \Theta)^2 +(2+\Theta)^2  \big)\Theta \leq \Theta
\end{equation} 
follows from \eqref{boundZ_2}.

Also, note that $\tilde{G_2}$ is a lower anti-triangular Hankel matrix with 1's on the main antidiagonal. Thus, $\tilde{G}_2\tilde{I}=I + E$, where $E$ is a nilpotent matrix with $E^n=0$ and $\| E \| \leq M$ for some $n\in\mathbb{N}$ and $M>1$. Define $F$ as follows.
\begin{displaymath}
F=f(E)=
\sum_{k=0}^{n-1} \frac{(-1)^{k}}{2k}\Big( \prod_ {j=1}^{k} \frac {1+2j}{2j} \Big) E^{k}
\end{displaymath} where $f(E)= \sqrt{(I+E)^{-1}}$. Note that $E=\tilde{G_2}\tilde{I} - I$. So,
\begin{displaymath}
\| I - Z_4 \| \leq \| I- F \| \leq \sum_{k=1}^{n-1} \| E \| ^{k} \le \sum_{k=1}^{n-1} (2M)^{k-1} \| E \| \le (n-1)(2M)^{n-2} \| E \|.
\end{displaymath}
By letting $C=(n-1)(2M)^{n-2}$ and  using \eqref{perturbationGP}, we get
\begin{displaymath}
\| I - Z_4 \| \leq C \| E \| = C \| E \tilde I \| =  C\| \tilde{G_2} -  P \| = C\| Z_3^{*}\tilde{G}_1Z_3 - Z_3^{*}PZ_3
\end{displaymath}
\begin{displaymath}
+Z_3^{*}PZ_3 - Z_3^{*}P +Z_3^{*}P -P\|\leq C  ( \|  Z_3^{*}(\tilde{G}_1-P)Z_3 \| + \| Z_3^{*}P(Z_3-I) \|
\end{displaymath}
\begin{displaymath}
 + \| Z_3- I \| \| P \| )\leq C(\| Z_3 \|^2 \| \tilde{G}_1-P \| +( \| Z_3\| +1)\| P \|  \| Z_3-I \|)
 \end{displaymath}
 \begin{displaymath}
  < \big( (1+ \Theta)^2 +(2+\Theta)^2  \big)\Theta \leq \Theta.
\end{displaymath}

Therefore,  we have the last inequality we needed
\begin{equation}\label{eq:result2}
\| I-Z_4 \| \leq \Theta.
\end{equation}


{\bf Step 5 (Stability of $N$).} Finally, we arrived at the desired conclusion for
\begin{displaymath}
(A,H)\overset{N}{\mapsto} (J,P),
\end{displaymath}
where $N=Z_1Z_2Z_3Z_4$.


To sum up, by using Lemma \ref{lemma2}, we see that $N- T_{0}$ is a small perturbation with
$
\| N - T_{0} \|  
\leq \Theta 
$
via \eqref{perturbationZ1T0}, \eqref{boundZ_2}, \eqref{boundZ_3} and \eqref{eq:result2}.
\end{proof}


\section{\bf Real Canonical Bases and Forms}
In  Theorem \ref{thm:focs1}, the real matrices $(A,H)$  were reduced to $(J,P)$, where, generally speaking, $J$ is  not real. Moreover, the affiliation matrix $T$ in $(A,H)\overset{T}{\mapsto}(J,P)$ is  not necessarily  real as well.

The following well-known  result (see, eg.,~\cite{GLR86} and the references therein) describes a purely real relation

\begin{equation}\label{eq:mapsR}
(A,H)\overset{R}{\mapsto}(J_R,P)
\end{equation}
where all five matrices are real.

\begin{proposition}\label{thm:real main}Let ${A} \in\mathbb{R}^ {n \times n}$  be a fixed real $H$-selfadjoint matrix for some self-adjoint real matrix $H$. Then there exists an invertible real matrix $R$ such that $(A,H)\overset{R}{\mapsto}(J_R,P)$. Here $J_R$ is a real Jordan form of $A$, that is 
	\begin{equation}\label{eq:real jordan}
	J_R= J(\lambda_1)\oplus \cdots \oplus J(\lambda_\alpha)\oplus \hat{J}_R(\lambda_{\alpha+1})\oplus \cdots \oplus\hat{J}_R(\lambda_\beta),
	\end{equation}
	where  $J(\lambda_k)$ (for $k=1,\ldots,\alpha$) are real  Jordan blocks  as before and for $k=~\alpha+~1,\ldots,\beta$ we have
		\begin{equation*}
	\hat{J}_R(\lambda_k)=
	{
		\left [ \begin{smallmatrix}
		\sigma_k & \tau_k   & 1        &      0 & 0    & \cdots&      & \cdots&  & &0 \\
		-\tau_k  & \sigma_k & 0        &      1 & 0    &     \ddots  &      &       &  & &  \\ 
		0& 0        & \sigma_k & \tau_k & 1    &\ddots &   \ddots   &       &  & &\vdots\\ 
		0 & 0   &\ddots    &\ddots  &\ddots&\ddots &\ddots&\ddots& &  \\
		&  \ddots        &  \ddots  & \ddots &\ddots  &\ddots&\ddots &\ddots &\ddots & & \vdots\\
		\vdots  &  &\ddots  &\ddots & \ddots&\ddots   & \ddots&\ddots &\ddots & \ddots&\\
		&  &  & \ddots&\ddots & \ddots    & \ddots& \ddots& \ddots & \ddots&  0\\
		&  &  &  &\ddots& \ddots&   0&\sigma_k & \tau_k&1 &0\\
		\vdots &  &  & & &   \ddots &\ddots &-\tau_k & \sigma_k & 0&1\\
		&  &  & & & &    \ddots &0 & 0 &\sigma_k &\tau_k\\[0.6em]
		0 &  & \ldots & &  &\ldots &  & 0& 0 &-\tau_k &\sigma_k\\
		\end{smallmatrix} \right]},
	\end{equation*} where $\lambda_k=\sigma_k+i\tau_k$.
	$P$ is a sip matrix of the form 
	\begin{displaymath}
	P=P_1 \oplus\cdots \oplus P_{\alpha} \oplus {P}_{\alpha+1} \oplus \cdots \oplus{P}_{\beta},
	\end{displaymath}
	where $P_k$ is a sip matrix $\epsilon_k \widetilde{I}$ of the same size as $J({\lambda_k})$ for $ k=1,\ldots,\alpha$ and a sip matrix $\widetilde{I}$ of the same size as  $\hat{J}_R({\lambda_k})$ for $k=\alpha+1,\ldots,\beta$ and $\epsilon_{k}=\pm{1}$ for $k=1,\ldots,\alpha$. 
\end{proposition}

\begin{definition}{\bf (Real Canonical Basis (RC))} Suppose that $A \in \mathbb{R}^{n \times n}$ and $J_R$ is the real Jordan form in \eqref{eq:real jordan}. We say that columns of matrix $R$ in \eqref{eq:mapsR} form a real canonical basis of $A$.
\end{definition}

The main result of this section is the Lipschitz stability of RC bases (see Theorem~{\ref{stabilityRC}}). To prove it we need the following lemma that provides a simple, yet useful connection between the {\it{i}}-FOCS and RC bases of the real pair of matrices $(A,H)$. As we will see, for an arbitrary pair $(A,H)$, this relation is carried over with the help of the same fixed matrix $S$.

\begin{lemma}\label{thm:focs-rc} Let $A,H \in \mathbb{R}^{n \times n}$, $A$ is $H$-selfadjoint, and the columns of
	\begin{equation}\label{eq:t}
	T=[T_1, \ldots, T_{\alpha} \lvert T_{\alpha+1}, \ldots, T_{\beta}] \text{  and }
	T_k=[Q_k \lvert i\overline{Q}_k], \quad k=\alpha+1, \ldots, \beta
	\end{equation}
	capture the {\it{i}}-FOCS basis of $(A,H)$. 
	
	Further, let $S=\diag({I_1,\ldots,I_{\alpha} | S_{\alpha+1},\ldots,S_{\beta}})$, where 
	
	\begin{displaymath}
	S_j=\frac{1}{\sqrt{2}}
	\left [ \begin{array}{ccccccccc}
	1&-i&0&0&\cdots& \cdots &0\\
	0&0&1&-i&0&\cdots & 0\\
	\vdots & \ddots   & \ddots& \ddots & \ddots & \ddots &\vdots\\
	\vdots & {}   & \ddots& \ddots & -i& 0 &0\\
	0&\cdots&\cdots& 0&0&1&-i\\
	-i&1&0&0&\cdots& \cdots &0\\
	0&0&-i&1&0&\cdots & 0\\
	\vdots & \ddots   & \ddots& \ddots & \ddots & \ddots &\vdots\\
	\vdots & {} & \ddots& \ddots & 1& 0 &0\\
	0& \cdots&\cdots& 0&0&-i&1\\
	\end{array} \right]\\.
	\end{displaymath}Then the matrix 
	
	\begin{equation}\label{eq:R}
	R=TS
	\end{equation}is real and its columns  
	
	\begin{equation}\label{eq:r}
	R=[T_1, \ldots, T_{\alpha} \lvert K_{\alpha+1}, \ldots, K_{\beta}]
	\end{equation}
	capture the RC basis of $A$.  
	Moreover,  $K_{j}=T_{j}S_j$ for $j=\alpha+1, \ldots, \beta$. 
	
	Conversely, let the columns of  $R$  capture the RC basis of $A$. Then the columns of $T=RS^{-1}$ capture the {\it{i}}-FOCS basis of $A$. Moreover, if we partition $T$ as in \eqref{eq:t}, then $T_j=R_jS_j^{-1}$. Finally, here is the explicit formula for $S_j^{-1}$:

	\begin{displaymath}
	S_j^{-1}=\frac{1}{\sqrt{2}}
	\left [ \begin{array}{ccccccccccccc}
	1&0&\cdots &0&i&0 &\cdots&0\\
	i&0&{}&\vdots&1&0&{}&\vdots\\
	0&1&\ddots&0 &0&i&\ddots&0\\
	0&i&\ddots&\vdots&  0 &1&\ddots&\vdots \\
	\vdots&\ddots &\ddots&0&\vdots&\ddots &\ddots&0\\
	\vdots&\ddots &\ddots&1&\vdots&\ddots &\ddots&i\\
	0&\cdots & 0&i&0&\cdots&0 & 1\\
	\end{array} \right].
	\end{displaymath}
	
\end{lemma}
\begin{proof}
	First, let us show that the columns of  matrix $R$ in \eqref{eq:R} capture the RC basis of $A$, that is
		\begin{displaymath}
	(A,H)\overset{R}{\mapsto}(J_R,P).
	\end{displaymath}  
	Since $(A,H)\overset{T}{\mapsto}(J,P)$, the above relation follows from
		\begin{equation}\label{eq:mapsS}
	({J},P)\overset{S}{\mapsto}({J}_R,P).
	\end{equation}
	
	We will now prove \eqref{eq:mapsS}. Without loss of generality we can consider just a single Jordan block, assuming that $\sigma(A)=\big\{\lambda, \overline{\lambda}\big\}$.
	The first relation $S^{-1}\hat{J}S=\hat{J}_R$ in \eqref{eq:mapsS} follows from		
	\begin{displaymath}\hat{J}S=\frac{1}{\sqrt{2}}
	{
		\left [ \begin{smallmatrix}
		\sigma+i\tau&-i\sigma+\tau&1 & -i &0&\dots&0&0\\
		0&0&\sigma+i\tau&-i\sigma+\tau&0&\dots&0 & 0 \\
		0&0& 0 & 0&1&\dots &0&0 \\
			0&0&0& 0 &\sigma+i\tau&\dots&0&0\\
			\vdots&\vdots&\vdots&\vdots &\vdots&\ddots&\vdots&\vdots\\
			0&0&0&0&0&\dots&1 & -i \\
			0&0& 0 & 0&0&\dots &\sigma+i\tau&-i\sigma+\tau \\
		-i\sigma-\tau&\sigma-i\tau&-i& 1 &0&\dots&0&0\\
		0&0&-i\sigma-\tau&\sigma-i\tau&0&\dots&0& 0 \\
		\vdots&\vdots&\vdots&\vdots &\vdots&\ddots&\vdots&\vdots\\
		0&0&0&0&0&\dots&-i & 1 \\
		0&0& 0 & 0 &0&\dots&-i\sigma-\tau&\sigma-i\tau \\
		\end{smallmatrix} \right]}=S\hat{J}_R.
	\end{displaymath}
	The second relation $S^{\ast}PS=P$ in  \eqref{eq:mapsS} follows from
	\begin{displaymath}
	S^{*}PS=P \Rightarrow PS^{*}PS=PP=I.
	\end{displaymath}
	Thus, we just need to show that $PS^{*}PS=I$.
	\begin{displaymath}PS^{*}PS=\frac{1}{2}\left [ \begin{smallmatrix}
	&&&&&&&\\
	0&0&\dots&0&i&0&0&\dots&1\\
	0&0&\dots&0&1&0&0&\dots&i\\
	0&0&\dots&i&0&0&0&\dots&0\\
	0&0&\dots&1&0&0&0&\dots&0\\
	\vdots&\vdots&\iddots&\vdots&\vdots&\vdots&\vdots&\iddots&\vdots\\
	0&i&\dots&0&0&0&1&\dots&0\\
	0&1&\dots&0&0&0&i&\dots&0\\
	i&0&\dots&0&0&1&0&\dots&0\\
	1&0&\dots&0&0&i&0&\dots&0\\
	&&&&&&&\\
	\end{smallmatrix} \right]\\
	\left [ \begin{smallmatrix}
	0&0&0&0&\dots&-i&1\\
	\vdots&\vdots&\vdots&\vdots&\iddots&\vdots&\vdots\\
	0&0&-i&1&\dots&0&0\\
	-i&1&0&0&\dots&0&0\\
	0&0&0&0&\dots&1&-i\\
	\vdots&\vdots&\vdots&\vdots&\iddots&\vdots&\vdots\\
	0&0&1&-i&\dots&0&0\\
	1&-i&0&0&\dots&0&0\\
	\end{smallmatrix} \right]\\=I.
	\end{displaymath}
	Secondly, we show that $R$ is real. We have
	\begin{displaymath}
	TS=[T_1, \ldots, T_{\alpha} | T_{\alpha+1}S_{\alpha+1}, \ldots, T_{\beta}S_{\beta}],
	\end{displaymath}
	where $T_k$ (for $k=1,\ldots, \alpha$) are real. Now, for $k=\alpha+1,\ldots, \beta$ we have
	\begin{displaymath}T_kS_k=[Q_k \lvert i\overline{Q}_k]S_k=
	[t_1, \ldots, t_{p_k} \lvert i\overline{t}_{1}, \ldots, i\overline{t}_{p_k} ]S_k= [t_1+\overline{t}_1, i(\overline{t}_1-t_1),t_2
	\end{displaymath}
	\begin{equation*}
	+\overline{t_2}, i(\overline{t}_2-t_2) , \ldots , t_{p_k}+\overline{t}_{p_k}, i(\overline{t}_{p_k} - t_{p_k})],
	\end{equation*}
	which is real as well. Thus, we have 
	\begin{displaymath}
	(A,H)\overset{R}{\mapsto}(J_{R},P).
	\end{displaymath}
	The converse follows from the relation $(J_R,P)\overset{S^{-1}}{\mapsto}(J,P)$, using similar argument.
	
\end{proof}

Now we are ready to prove the main result of this section.


\begin{theorem}{\bf (Stability of RC basis)}\label{stabilityRC} Let $A_0 \in \mathbb{R}^{n \times n}$ be a fixed $H_0$-selfad\-joint matrix where $H_0^{T}=H_0 \in \mathbb{R}^{n \times n}$ is an invertible matrix. Let
	\begin{displaymath}
	(A_0, H_0)\overset{R_0}{\mapsto}(J_{R},P)
	\end{displaymath}
	be the mapping from Proposition \ref{thm:real main}, that is the columns of $R_0$ capture the $\it{RC}$ basis of $A_0$.
	There exist constants $K,\delta>0$ (depending on $A_0$ and $H_0$ only) such that the following assertion holds. For any real pair $(A,H)$ such that $A$ is $H$-selfadjoint and $A$ is similar to $A_0$ and
		\begin{equation*}
	\| A-A_0 \| + \| H-H_0 \| < \delta.
	\end{equation*}
	Then, there exists a matrix $R$ such that
	\begin{displaymath}
	(A, H)\overset{R}{\mapsto}(J_{R},P),
	\end{displaymath}whose columns capture the RC basis of A, and such that \begin{equation}\label{lbperwithconstant}
	\| R-R_0 \| \leq {K}(\| A-A_0 \| + \| H-H_0 \|).
	\end{equation}
	\end{theorem}
\begin{proof}
Here is the diagram of our proof.

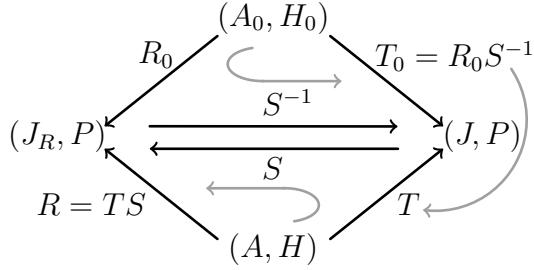
\begin{center}
\definecolor{aqaqaq}{rgb}{0.6274509803921569,0.6274509803921569,0.6274509803921569}
\begin{tikzpicture}[line cap=round,line join=round,x=1.5cm,y=1.5cm]
\clip(-6,0.7) rectangle (-1.2,3.5);
\draw (-4.1,3.4) node[anchor=north west] {$(A_0,H_0)$};
\draw (-4,1.35) node[anchor=north west] {$(A,H)$};
\draw (-5.95,2.35) node[anchor=north west] {$(J_R,P)$};
\draw (-2.1,2.35) node[anchor=north west] {$(J,P)$};
\draw [->,line width=1pt] (-4,3) -- (-5,2.2);
\draw [->,line width=1pt] (-3,3) -- (-2,2.2);
\draw [->,line width=1pt] (-4,1.2) -- (-5,2);
\draw [->,line width=1pt] (-3,1.2) -- (-2,2);
\draw [->,line width=1pt] (-4.6,2.2) -- (-2.4036363636363642,2.1981818181818187);
\draw [->,line width=1pt] (-2.4,2) -- (-4.6,2);
\draw [<-,shift={(-2.067706863648725,2.2583274232818553)},line width=1pt,color=aqaqaq]  plot[domain=-1.7076660690284706:0.5904674383413426,variable=\t]({1*0.8089711881931181*cos(\t r)+0*0.8089711881931181*sin(\t r)},{0*0.8089711881931181*cos(\t r)+1*0.8089711881931181*sin(\t r)});
\draw (-2.7,3.05) node[anchor=north west]  {$T_0=R_0S^{-1}$};
\draw (-3.7,2.02) node[anchor=north west] {$S$};
\draw (-3.7,2.62) node[anchor=north west] {$S^{-1}$};
\draw (-2.5,1.7) node[anchor=north west] {$T$};
\draw (-5.7,1.7) node[anchor=north west] {$R=TS$};
\draw (-4.8,3.05) node[anchor=north west] {$R_0$};
\draw [shift={(-3.6125,2.70625)},line width=1pt,color=aqaqaq]  plot[domain=1.86:4.75,variable=\t]({0.01+1*0.3065049958809803*cos(\t r)+0*0.3065049958809803*sin(\t r)},{-0.0020+0*0.3*cos(\t r)+1*0.15*sin(\t r)+0.0463});
\draw [->,line width=1pt,color=aqaqaq] (-3.6,2.6) -- (-2.9,2.6);
\draw [shift={(-3.3875,1.49375)},line width=1pt,color=aqaqaq]  plot[domain=-1.28:1.6,variable=\t]({-0.02+1*0.3065049958809809*cos(\t r)+0*0.3065049958809809*sin(\t r)},{0.007+0*0.3*cos(\t r)+1*0.15*sin(\t r)});
\draw [->,line width=1pt,color=aqaqaq] (-3.4,1.65) -- (-4.1,1.65);
\end{tikzpicture}
\end{center}

Let the columns of $R_0$ capture the RC basis of $A_0$, i.e. 
	\begin{equation*}\label{eq:mapR_0}
	(A_0, H_0)\overset{R_0}{\mapsto}(J_{R},P). 
	\end{equation*} 
	Then, by Lemma {\ref{thm:focs-rc}} columns of
		\begin{equation}\label{eq:s^{-1}}
	T_0=R_0S^{-1}
	\end{equation}
	capture the columns of {\it{i}}-FOCS basis of $A_0$, that is $(A_0, H_0)\overset{T_0}{\mapsto}(J,P)$.
	
	Theorem~\ref{thm:main2} says that there exists a matrix $T$, capturing the {\it{i}}-FOCS basis of $A$, such that $(A, H)\overset{T}{\mapsto}(J,P)$ and for some $\hat{K}$
	\begin{equation}\label{eq:bound}
	\| T-T_0 \|  \leq\hat{K}( \| A- A_{0} \| + \| H- H_{0} \| ). 
	\end{equation}
	Let us define 
$	R=TS$.
	Then, 
	\begin{equation}
	(A, H)\overset{R}{\mapsto}(J_{R},P). 
	\end{equation}
	To obtain \eqref{lbperwithconstant} we observe that
	\begin{displaymath}
	\| R-R_0 \| = \| TS -T_0S \| \leq \| S \| \| T-T_0 \| \leq {K}( \| A- A_{0} \| + \| H- H_{0} \| ) 
	\end{displaymath} 
	follows from \eqref{eq:s^{-1}}, \eqref{eq:bound} and \eqref{eq:r} by denoting ${K}=\| S\|\hat {K}$. Thus, we complete the proof of the Theorem {\ref{stabilityRC}}.
\end{proof}


\section{\bf {Weak Affiliation Case} }

In previous sections, we considered a strong affiliation relation $(A,H)\overset{T}{\mapsto}(J,P)$ where matrices $A$ and $J$ were similar. Here we weaken this condition and only require  $A$ and $J$ to have the same Jordan structure. Here is a formal definition. 
		
\begin{definition}
		{\bf{(Same Jordan Structure)}} Denote by $\sigma{(A_{0})}$ and $\sigma{(A)}$ the sets of all eigenvalues of $A_{0}$ and $A$, respectively. Matrices $A_{0}$ and $A$ are said to have the {\bf same Jordan structure} if there is a bijection  $f:\sigma{(A_{0})} \rightarrow \sigma {(A)} $ such that if 
		\begin{displaymath}
		\mu=f(\lambda) 
		\end{displaymath}
		then $\lambda\in\sigma(A_0)$ and $\mu\in\sigma(A)$ have the same Jordan block sizes or, equivalently, if there is  an invertible matrix \(S\) such that $A_0$  and $S^{-1}AS$ have the same Jordan bases.
\end{definition}

We presented the weak affiliation in Section 1. Let $A_0$ be $H_0$-selfadjoint. To remind the reader, a matrix $S$ is called a {\bf weak affiliation matrix} of the quadruple $(A_0, H_0, A, H)$ if
	\begin{displaymath}
	(A,H)\overset{S}{\mapsto}_{w}(A_0,H_0),
	\end{displaymath}where matrices $A_0$ and $A$ have the same Jordan structure. In this case $(A_0,H_0)$ and $(A,H)$ are called {\bf weakly affiliated}. 

In \cite{BOP}  authors proved not only Proposition {\ref{thm:stabilityFO}}, but also its ''weak'' version, which we formulate next.


\begin{proposition}\label{thm:stabilityweak} Let $ A_0 \in \mathbb{C}^{ n\times n}$ be a fixed $H_0$-selfadjoint matrix and let columns of this matrix capture the FO-basis, that is 
	\begin{displaymath}
	(A_0,H_0)\overset{T_0}{\mapsto}(J,P).
	\end{displaymath}
	There exist constants $K,\delta > 0$ (depending on $ A_{0}$ and $ H_0 $ only) such that  the following assertion holds. For any $H$-selfadjoint matrix $A$  that  has the same Jordan structure as $A_0$ and
	\begin{displaymath}
	\| A-A_0\| + \| H-H_0\| < \delta,
	\end{displaymath}
	there exists a matrix $T$ whose columns capture the FO-basis of $A$, i.e. 
	\begin{displaymath}
	(A,H)\overset{T}{\mapsto}_{w}(J,P),
	\end{displaymath}and such that 
	\begin{equation*}\label{TT0weak}
	\| T - T_{0} \| \leq K(\| A-A_0 \| + \| H-H_0 \|).
	\end{equation*}
\end{proposition}
Now, we are going to  introduce the  ''weak'' versions of Theorem {\ref{thm:main2}} and  Theorem~{\ref{stabilityRC}}.


\begin{theorem}\label{weakstability}{\bf{(Weak Stability of $\gamma$-FOCS basis)}} Let $ A_0 \in \mathbb{R}^{ n\times n}$ be a fixed $H_0$-selfadjoint, where $H_0\in \mathbb{R}^{n \times n}$ is a Hermitian matrix, and let the columns of this matrix capture the FOCS-basis, that is 
	\begin{displaymath}
	(A_0,H_0)\overset{T_0}{\mapsto}(J,P).
	\end{displaymath}
	Then, there exist constants $K,\delta > 0\ (depending\ on \ A_{0}\ and \ H_0 \ only)$ such that  the following assertion holds. For any $H$-selfadjoint matrix $A$ such that $A$,~$H \in~\mathbb{R}^{n \times n}$, invertible $A$ has the same Jordan structure as $A_0$, and
	\begin{displaymath}
	\| A-A_0\| + \| H-H_0\| < \delta,
	\end{displaymath}
	there exists an invertible matrix $N \in \mathbb{R}^{n \times n}$ whose columns capture the FOCS-basis of $A$, that is  
	\begin{displaymath}
	(A,H)\overset{N}{\mapsto}_{w}(J,P),
	\end{displaymath}and such that 
	\begin{equation*}
	\| N - T_{0} \| \leq K(\| A-A_0 \| + \| H-H_0 \|).
	\end{equation*} 
\end{theorem}

\begin{theorem}{\bf{(Weak Stability of  RC basis)}}\label{weakstabilityRC} Let $A_0 \in \mathbb{R}^{n \times n}$ be a fixed $H_0$-selfadjoint matrix where $H_0^{T}=H_0 \in \mathbb{R}^{n \times n}$. Let
	\begin{displaymath}
	(A_0, H_0)\overset{R_0}{\mapsto}(J_{R_0},P)
	\end{displaymath}
	be the mapping of Theorem~\ref{thm:real main} that is the columns of $R_0$ capture the $\it{RC}$ basis of $A_0$.
	There exist constants $K,\delta>0$ $($depending on $A_0$ and $H_0$ only$)$ such that the following assertion holds. For any real pair $(A,H)$ such that $A$ is $H$-selfadjoint and $A$ has the same Jordan structure as $A_0$ with
	\begin{equation*}
	\| A-A_0 \| + \| H-H_0 \| < \delta,
	\end{equation*}
 there exists a matrix $R$ such that
	\begin{displaymath}
	(A, H)\overset{R}{\mapsto}(J_{R},P),
	\end{displaymath}  whose columns capture the RC basis of $A$, and  
	\begin{equation*}
	\| R-R_0 \| \leq {K}(\| A-A_0 \| + \| H-H_0 \|).
	\end{equation*}
	Moreover, there exists a real weak affiliation matrix $F$, i.e.
$
	(A, H)\overset{F}{\mapsto}_w(A_0,H_0)$ such that 
	\begin{equation*}
	\| I-F \| \leq {K}(\| A-A_0 \| + \| H-H_0 \|).
	\end{equation*}
	
\end{theorem}
{
The proofs of  Theorem~\ref{weakstability} and Theorem {\ref{weakstabilityRC}} follow exactly the same lines as the proofs of Theorem~{\ref{thm:main2}} and Theorem~{\ref{stabilityRC}}, respectively. For example, in the proof of Theorem {\ref{thm:main2}}  we used Proposition~{\ref{thm:stabilityFO}}. Now, we apply the result of Proposition {\ref{thm:stabilityweak}} instead, in order to get the proof of Theorem~\ref{weakstability}. All the rest remains the same.}




\begin{thebibliography}{Name YY}
	\bibitem{BOP}T. Bella,  V. Olshevsky, U. Prasad, {\it Lipschitz stability of canonical Jordan bases of $H$-selfadjoint matrices under structure-preserving perturbations}, Linear Algebra and its Applications, {\bf 428}, 8--9, 2008,  2130--2176
	\bibitem{DMO} {S. Dogruer Akgul}, {A. Minenkova}, {V. Olshevsky,} {\it Existence of flipped orthogonal conjugate symmetric Jordan canonical bases for real $H$-selfadjoint  matrices},  	\href{https://arxiv.org/pdf/2203.09877}{arXiv:2203.09877}.

	
	
	
	\bibitem{GK78}   I. Gohberg,  M.A. Kaashoek, {\sl  Unsolved problems in matrix and operator theory. I. Partial multiplicities and additive perturbations}, {Integral Equations and Operator     Theory}, {\bf 1}, 1978, 278--283.
	
	\bibitem{GLR83} I. Gohberg, P. Lancaster, L. Rodman, {\sl
		Matrices and Indefinite Scalar Products},  Birkh\"auser, 1983, xvii+374 pp. 
	
	\bibitem{GLR86} I. Gohberg, P. Lancaster, L. Rodman, {\sl Invariant
		Subspaces of Matrices with Applications}, Canadian Mathematical Society Series of Monographs and Advanced Texts. A Wiley-Interscience Publication. John Wiley\& Sons, Inc., New York, 1986, xviii+692~pp.
	
	\bibitem{GLR05} I. Gohberg, P. Lancaster, L. Rodman, {\sl Indefinite
		Linear Algerba and Applications}, Birkh\"auser, 2005, xii+357 pp.
	
	\bibitem{GR86} I. Gohberg, L. Rodman,  {\sl On distance between
		lattices of invariant subspaces of matrices,}  Linear Algebra and
	its Applications, {\bf76}, 1986, 85--120.
	
	\bibitem{K66} T. Kato,  {\sl Perturbation theory for linear operators}, Die Grundlehren
	der mathematischen Wissenschaften, Band 132 Springer-Verlag New York,
	Inc., New York, 1966, xix+592 pp.
	
	
	\bibitem{M63} A.I. Mal'cev, {\sl Foundations of Linear Algebra}, W. H. Freeman\& Co., San Francisco, Calif.-London 1963 xi+304 pp. 
	(Translation from Russian).
	
	\bibitem{MBO97} J. Moro, J.V. Burke, M.L. Overton,
	{\sl On the Lidskii-Vishik-Lyusternik perturbation theory for eigenvalues of matrices with arbitrary Jordan structure}, SIAM J. Matrix Anal. Appl., {\bf18}, 4, 1997,  793--817. 
	
	
	\bibitem{MP80}   A. Markus, E. Parilis, {\sl The change of the
		Jordan structure
		of a matrix under small perturbations}, { Mat. Issled.},
	{\bf 54} (1980),  98--109 ({\em in Russian}),
	{\sl English translation}: {Linear Algebra
		Appl.}, {\bf 54}, 1983, 139--152.
	
	
	\bibitem{O91} V. Olshevsky, {\sl  The change of the Jordan structure of G-selfadjoint operators and selfadjoint operator functions under small perturbations}, (in
	Russian), Izvestia Akad. nauk U.S.S.R.,   {\bf54}, 5, 1990,  1021--1048,
	English translation: AMS, Math. U.S.S.R. Izvestia, {\bf37}, 2, 1991, 371--396.
	
	
	\bibitem{R06} L. Rodman, {\sl Similarity vs unitary similarity and perturbation analysis
		of sign characteristics: Complex and real indefinite inner
		products}, Linear Algebra and its Applications, {\bf416}, 2--3, 2006,  945--1009.
	
	\bibitem{SS90} G.W. Stewart, J.G. Sun, {\sl Matrix Perturbation Theory}, Academic Press Inc., Boston, MA, 1990, xvi+365 pp. 
	
	\bibitem{W68}  K. Weierstrass, {\sl Zur Theorie der quadratischen und bilineareen Formen}, Monatsber. Akad. Wiss., Berlin, 1868, 310--338.
	
	\bibitem{W95}  K. Weierstrass, {\sl Mathematische Werke, Zweiter Band.  Abhandlungen II}, Berlin, Mayer \& Muller, 1895,  19--44.
	
	
\end{thebibliography}


\end{document}